\documentclass[a4paper,12pt]{article}

\usepackage[T1]{fontenc}
\usepackage[english]{babel}
\usepackage{amsmath,amssymb,amsfonts,amsthm,mathtools}
\usepackage{mathrsfs}
\usepackage{dirtytalk}
\usepackage{longtable}
\usepackage[hidelinks]{hyperref}
\usepackage{xurl}
\usepackage{booktabs}
\usepackage{algorithm}
\usepackage[noend]{algpseudocode}
\usepackage{comment}

\algnewcommand{\IfThen}[2]{%
  \State \algorithmicif\ #1\ \algorithmicthen\ #2%
}

\algnewcommand{\IfThenElse}[3]{%
  \State \algorithmicif\ #1\ \algorithmicthen\ #2\ \algorithmicelse\ #3%
}

\theoremstyle{plain}
\newtheorem{theorem}{Theorem}
\newtheorem{corollary}[theorem]{Corollary}
\newtheorem{lemma}[theorem]{Lemma}

\theoremstyle{definition}
\newtheorem{definition}{Definition}

\theoremstyle{definition}

\newcommand{\Aut}{\ensuremath{\mathrm{Aut}}}

\begin{document}

\title{Equivalence of Complex Hadamard Matrices}

\author{
Patric R. J. \"Osterg\aa rd and Tuomo Valtonen\\
Department of Information and Communications Engineering\\
Aalto University School of Electrical Engineering\\
P.O.\ Box 15600, 00076 Aalto, Finland\\
Email: {\tt \{patric.ostergard,tuomo.valtonen\}@aalto.fi}}

\date{}

\maketitle

\begin{abstract}
Symmetry in the context of equivalence or isomorphism is a
fundamental and natural concept in any study of discrete 
structures. Symmetries are also important for non-discrete
structures, but their treatment can be more challenging 
and is perhaps therefore often overlooked. This holds for
many studies of complex Hadamard matrices, that is, matrices
with unimodular complex entries satisfying the equation 
$HH^{\dagger} = nI$, where $H^{\dagger}$ is the conjugate
transpose of $H$. In the current work, equivalence 
of complex Hadamard matrices is considered, and algorithms
for determining equivalence of matrices and the automorphism 
group of a matrix are presented. The algorithms are used to establish the automorphism group of a large number of complex Hadamard matrices from the literature. 
\end{abstract}


\section{Introduction}

A \emph{complex Hadamard matrix of order $n$} is an $n \times n$
matrix $H = (h_{ij})$ with unimodular complex entries, that is, 
$h_{ij} \in S \coloneqq \{z \in \mathbb{C} : |z| = 1\}$ 
for all $1\leq i,j\leq n$, such that $HH^{\dagger}=nI$, where $H^{\dagger}$ is
the conjugate transpose of $H$ and $I$ is the identity matrix.
For a survey of complex Hadamard matrices, see \cite{TZ06}.

Complex Hadamard matrices with entries that are $q$th roots of
unity for some positive integer $q$ are called 
\emph{Butson-type Hadamard matrices}. (Real) \emph{Hadamard matrices}
are Butson-type Hadamard matrices with $q=2$. Equivalence of 
Butson-type Hadamard matrices is discussed in \cite{O22};
the definitions and the terminology in the current work
are in line with that study. 

We define the 
group $\mathcal{G}_n$ as the matrix group 
that is the subgroup of $\mbox{GL}(n,\mathbb{C})$ whose elements 
are the matrices in which each row and column contains exactly one 
element of $S$, and all other entries are 0. We further let the group 
$\mathcal{G}_n \times \mathcal{G}_n$ 
act on an $n \times n$ matrix $M$ by $(X,Y)\cdot M = XMY^{\mathrm{\dagger}}$.
The matrix obtained by complex conjugation of each entry of 
a matrix $A$ is denoted by ${\overline {A}}$.

\begin{definition}
\label{def:eq}
Two $n\times n$ matrices with entries from $S$, 
$H_1$ and $H_2$, are said to be \emph{equivalent},
denoted by $H_1\cong H_2$, if $H_1 = XH_2Y^{\dagger}$ for
some $(X,Y) \in \mathcal{G}_n \times \mathcal{G}_n$
and they are said to be \emph{Hadamard equivalent} if
$H_1\cong H_2$ or $H_1\cong \overline{H_2}$.
\end{definition}

The convention in the field is not to include operations involving transposition, $H^{\mathrm T}$ and $H^{\dagger}$, in any concept of equivalence.

Symmetries of complex matrices can be studied in the context of 
mappings from matrices onto themselves, which leads to the concepts
of automorphisms and automorphism groups. We shall get back to a
formal definition and treatment of these later.

Whereas a consideration of equivalence and automorphisms is present in 
virtually all studies of constructing and classifying real Hadamard
matrices and Butson-type Hadamard matrices,  
almost no attention has been paid to these in studies of complex Hadamard
matrices. Probable reasons for this are the challenges due to 
non-discreteness and the scenario of typically considering families 
of matrices rather than individual matrices. In the current work
we shall deal with these issues and, in particular, develop algorithms for determining equivalence of complex Hadamard matrices and for finding the automorphism group.

The paper is organized as follows. In Section~\ref{preliminaries}, we present the necessary preliminaries. In Section~\ref{algs}, we present algorithms for determining equivalence and for finding automorphism groups. In Section~\ref{inv}, we show how to apply invariants to improve the performance of the algorithms. In Section~\ref{SH}, we present algorithms for determining whether there exists a symmetric matrix or a Hermitian matrix equivalent to a given complex Hadamard matrix. In Section~\ref{implementation}, we discuss practicalities related to the implementations of the algorithm.  Finally, in Section~\ref{S3}, we discuss applications of the constructed algorithms and utilize the implementations to compute the automorphism groups of a large number of (families of) complex Hadamard matrices that appear in the literature.
\section{Automorphisms and Equivalence}
\label{S2}

Before discussing algorithms, we need to have a
closer look at some details of the theory and introduce some 
further concepts.

\subsection{Preliminaries}
\label{preliminaries}

Let us start by considering the group 
action utilized in Definition~\ref{def:eq} and giving a precise definition
of automorphism. One important fact is that the action 
of $\mathcal{G}_n \times \mathcal{G}_n$ is not faithful as
$(sI,sI)\cdot M = sIM\overline{s}I = M$ for any $s \in S$. The
kernel 
\begin{equation*}
\label{eq:kernel}
K := \{(sI,sI) : s \in S\} 
\end{equation*}
is obviously infinite, but can
be factored out as a normal subgroup.

Let $H$ be a complex Hadamard matrix. Consider the elements
$(X,Y) \in \mathcal{G}_n \times \mathcal{G}_n$ for
which $H = XHY^{\dagger}$. Those elements form a subgroup
$G$ of $\mathcal{G}_n \times \mathcal{G}_n$, and $G/K$ is
said to be the \emph{automorphism group} of $H$. The elements
of $G/K$ are called \emph{automorphisms} and subgroups of 
$G/K$ are \emph{groups of automorphisms}. 
The corresponding definitions for
Hadamard equivalence are analogous.

A \emph{dephased} complex Hadamard matrix is a matrix in which all 
elements in the first row and first column are 1. For other types
of Hadamard matrices, the term \emph{normalized} is commonly used
for this property. It is straightforward to turn an arbitrary
complex Hadamard matrix into an equivalent dephased one: Multiply rows and columns one by one with the conjugate of the
element in the first entry.

\begin{theorem}
\label{theorem1}
Let $H$ be a complex Hadamard matrix. The pair of diagonal matrices $(D_1,D_2) \in \mathcal{G}_n \times \mathcal{G}_n$ for which $(D_1,D_2)\cdot H$ is dephased is unique up to multiplication by elements in the kernel.

\end{theorem}

\begin{proof}
Consider a matrix $H$ and multiply column vector $i$ with $c_i$ 
and row vector $i$ with $r_i$, for $1 \leq i \leq n$. 
To get a dephased matrix, we
must have 
\[
r_1c_1h_{1,1} = 1,\ r_1c_2h_{1,2}=1,\ldots,\
r_1c_nh_{1,n} = 1,\ r_2c_1h_{2,1}=1,\ldots,\
r_nc_1h_{n,1}=1.
\]
We then get a solution
\[
\begin{array}{l}
c_1 = \overline{r_1h_{1,1}},\
c_2 = \overline{r_1h_{1,2}},\ldots,\
c_n = \overline{r_1h_{1,n}},\\
r_1 = r_1,\
r_2 = r_1h_{1,1}\overline{h_{2,1}},\ldots,\
r_n = r_1h_{1,1}\overline{h_{n,1}},
\end{array}
\]
\noindent 
which is a function of $r_1$. Equivalently, $H$ is dephased by the action of any pair of diagonal matrices
\[
(r_1D_1,r_1D_2)=(r_1I,r_1I)(D_1,D_2),
\]
where $D_1=\operatorname{diag}(\overline{h_{1,1}},\overline{h_{2,1}},\ldots,\overline{h_{n,1}})$,
$D_2=\operatorname{diag}\bigl(1,\overline{h_{1,1}}h_{1,2},\ldots,\overline{h_{1,1}}h_{1,n}\bigr)$,
and $r_1\in S$.
\end{proof}

With the notation in the proof of Theorem~\ref{theorem1} we define a dephasing operator $\mathscr{D}$ by $\mathscr{D}(H)\coloneqq D_1HD_2^\dagger.$

Every element $M \in \mathcal{G}_n$ can be written uniquely as $M = DP$, where $D$ is
a diagonal matrix with unimodular entries and $P$ is a permutation matrix. We have a group homomorphism $\varphi$ from the automorphism group $G/K$ of a complex Hadamard matrix $H$ to the group $S_n \times S_n$, where $S_n$ is a symmetric group of degree $n$, given by $(D_1P_1,D_2P_2)K\mapsto (P_1,P_2)$. We set $\operatorname{Aut}(H):=\varphi(G/K)$.

\begin{theorem}
\label{autt}
Let $H$ be a complex Hadamard matrix with automorphism group $G/K$. The groups $G/K$ and $\operatorname{Aut}(H)$ are isomorphic.
\end{theorem}
\begin{proof}
    
    The map $\varphi$ is a group homomorphism and surjective to its image, so we only need to show that $\varphi$ is injective, that is, ker$(\varphi)=\{K\}$. Let $(D_1P_1,D_2P_2)K\in \text{ker}(\varphi).$ As $\varphi((D_1P_1,D_2P_2)K)=(I,I)$ we must have $(P_1,P_2) = (I,I)$, which gives an element of the form $(D_1,D_2)K$. Let $(D'_1, D'_2)$ be a pair of diagonal matrices with unimodular entries such that $(D_1',D'_2)\cdot H$ is dephased. Then $(D_1'D_1,D_2'D_2)\cdot H$ is dephased as $(D_1,D_2)K$ is an automorphism. By Theorem~\ref{theorem1}, $(D_1'D_1,D_2'D_2)\in (D_1',D'_2)K$ implying $(D_1,D_2)K =  K$. Therefore, ker$(\varphi)=\{K\}$.
\end{proof}

\begin{corollary}
The automorphism group of a finite complex Hadamard
matrix is finite.
\end{corollary}
\begin{proof}
    By Theorem~\ref{autt}, the automorphism group is isomorphic to a subgroup of a finite group, so it is finite.
\end{proof}

In the sequel, it is in many places more convenient
to deal with permutations $\pi$ over the set $\{1, 2, \ldots, n\}$ rather than $n \times n$
permutation matrices $P$. For $\pi\in S_n$ we denote by $P_\pi$ the corresponding permutation
matrix, defined by
\[
(P_\pi)_{ij} =
\begin{cases}
1, & \text{if } i = \pi(j),\\
0, & \text{otherwise}.
\end{cases}
\]
With this convention, we have
\[
(P_{\pi_1},P_{\pi_2})\cdot H = P_{\pi_1} H P_{\pi_2}^\dagger
  = P_{\pi_1} H P_{\pi_2}^T
  = (h_{\pi_1^{-1}(i),\pi_2^{-1}(j)})_{i,j}=(\pi_1,\pi_2)\cdot H.
\]
We refer to elements of $\Aut(H)$ simply as automorphisms of $H$. Depending on the context, we may regard an automorphism as a pair 
of permutations $(\pi_1,\pi_2)$, as the corresponding pair of 
permutation matrices $(P_{\pi_1},P_{\pi_2})$, or as an element of $G/K$. In terms of permutations, the group
$\operatorname{Aut}(H)$ contains exactly the permutations
$(\pi_1,\pi_2)\in S_n\times S_n$ such that $$\mathscr{D}((\pi_1,\pi_2)\cdot H)=\mathscr{D}(H).$$

\subsection{Algorithms}
\label{algs}

Algorithms for handling equivalence of real and Butson-type Hadamard
matrices can be found in the literature \cite{EFO15,LOS20,LSO13, M79},
so we want to focus on general algorithms for matrices with 
arbitrary unimodular entries.

The equivalence problem can be solved by computing canonical
forms of complex Hadamard matrices. Two matrices are equivalent
exactly when their canonical forms coincide. The automorphism group can be computed by a similar procedure, with some extra bookkeeping.

In this section, we outline algorithms for both of these tasks. All algorithms presented are high-level sketches and can be made more efficient using invariants, which we will discuss in Section~\ref{inv}. The descriptions omit details required for practical implementations, some of which we will address in Section~\ref{implementation}. We refer to \cite{CombAlgs} for the relevant background on computations with groups.

We say that two complex Hadamard matrices, $H_1$ and $H_2$, are \emph{permutation equivalent} if there exist permutation matrices $P_1$ and $P_2$ such that $H_1=(P_1,P_2)\cdot H_2=P_1H_2P_2^T$. We define $\sigma_{a,b}:=((1\;a), (1\; b))$ and $\mathscr{D}_{a,b}(H):=\mathscr{D}(\sigma_{a,b}\cdot H)$.

\begin{lemma}
\label{ab}
    Let $H_1$ be a dephased complex Hadamard matrix and $H_2$ a matrix such that $H_1\cong H_2$. There exist $a,b\in\{1,\ldots,n\}$ such that $H_1$ and $\mathscr{D}_{a,b}(H_2)$ are permutation equivalent.
\end{lemma}
\begin{proof}
    First, note that for any permutations $\pi_1,\pi_2 \in S_n$ such that $\pi_1(1)=\pi_2(1)=1$ we have that 
    \begin{equation}
    \label{commutes}
      \mathscr{D}((\pi_1,\pi_2)\cdot H)= ( h_{1,1} \overline{h_{\pi_1^{-1}(i),1}} \overline{h_{1,\pi_2^{-1}(j)}} h_{\pi_1^{-1}(i),\pi_2^{-1}(j)} )_{i,j}=(\pi_1,\pi_2) \cdot \mathscr{D}(H).  
    \end{equation}
    As $H_1\cong H_2$, there exist permutation matrices $P_{\pi_1}, P_{\pi_2}$ and diagonal matrices of unimodular entries $D_1,D_2$ such that $H_1=D_1P_{\pi_1} H_2(D_2P_{\pi_2} )^\dagger$. Let $a= \pi_1^{-1}(1)$ and $b = \pi_2^{-1}(1)$. Because $(1\ a)$ maps $1$ to $a$ and $\pi_1$ maps $a$ to $1$, the composition $\pi_1(1\ a)$ leaves $1$ fixed. The same holds for $\pi_2(1\;b)$ and
   $$
    H_1 = D_1P_{\pi_1} H_2(D_2P_{\pi_2} )^\dagger
    =\mathscr{D}(P_{\pi_1} H_2P_{\pi_2}^T)
    =(P_{\pi_1} P_{(1\;a)})\mathscr{D}_{a,b}(H_2)(P_{\pi_2} P_{(1\;b)})^T,    
$$
where the second equality follows from the uniqueness of the dephased form in Theorem~\ref{theorem1}, and the third from equation \eqref{commutes}.
\end{proof}

Let us first consider the problem of recognizing when two dephased complex Hadamard matrices are permutation equivalent. The standard way to treat combinatorial problems of this kind is to encode the objects in question as colored graphs \cite[pp.~83--88]{K06}. Given an $n\times n$ matrix $A$, we can construct a vertex- and edge-colored graph on $2n$ vertices by introducing a vertex $r_i$ for each row, a vertex $c_j$ for each column, and an edge $\{r_i,c_j\}$ for each entry $a_{ij}$. Then, we assign one common color to all row vertices, another common color to all column vertices, and color edges with the value of the corresponding entry. We call the constructed graph a \emph{permutation graph} of $A$. The proof of the following result is rather straightforward and is omitted.

\begin{lemma}
\label{graph_lemma}
    Two complex Hadamard matrices, $H_1$ and $H_2$, are permutation equivalent if and only if they have isomorphic permutation graphs.
\end{lemma}

The automorphisms of the permutation graph are exactly the automorphisms of the corresponding matrix when we allow only row and column permutations. Isomorphic graphs have identical canonical forms. By Lemma~\ref{graph_lemma}, the canonical form of the permutation graph immediately provides us with the canonical form of the corresponding matrix with respect to permutation equivalence.

For computing the canonical form of a graph and for finding the automorphisms of a graph, we utilize external algorithms CanonGraph and AutGen. In practice, we use \emph{nauty}~\cite{MP14} to perform both of these external tasks. Since \emph{nauty} only handles vertex-colored graphs, we first transform the vertex- and edge-colored permutation graph into a vertex-colored graph. From this transformed graph, we can recover both the automorphisms and the canonical form of the original graph. Such a conversion is described in \cite[p.~64]{Nauty_User_Guide}. The resulting vertex-colored graph has 
$O(n\log k)$ vertices, where $k$ is the number of distinct edge colors.

The CanonGraph algorithm returns a pair of permutations $\pi_1$ and $\pi_2$ such that applying $\pi_1$ to the row vertices and $\pi_2$ to the column
vertices yields the canonical form of the graph. This pair of permutations
also uniquely determines the action on all other vertices. The AutGen algorithm returns generators of the automorphism group of the input graph. These automorphisms are also given in the form of pairs of permutations 
acting on rows and columns. Additionally, PermGraph will denote a function that returns the permutation 
graph of the input matrix. These algorithms are combined in Algorithm~\ref{alg_caneq} 
to compute the canonical form of a complex Hadamard matrix $H$ with respect to permutation equivalence.

\begin{algorithm}
\caption{CanPerm$(H)$}
\label{alg_caneq}
\begin{algorithmic}
\State \textbf{external} PermGraph($H$), CanonGraph($G$)
\State $G \gets$ PermGraph($H$)
\State $(\pi_1,\pi_2) \gets $ CanonGraph($G$)
\State \textbf{return} ($(\pi_1,\pi_2),(\pi_1,\pi_2)\cdot H)$
\end{algorithmic}
\end{algorithm}

Utilizing Lemmas~\ref{ab} and \ref{graph_lemma}, we can construct a canonical form of a complex Hadamard matrix with respect to the definition of equivalence of complex Hadamard matrices given in Definition~\ref{def:eq}. Let $H$ be a complex Hadamard matrix. By Lemma~\ref{ab}, the collection $\{\mathscr{D}_{a,b}(H)\}_{1 \leq a,b \leq n}$ contains all possible dephased forms up to permutation equivalence. Constructing permutation graphs for all those matrices and canonizing them, we obtain a collection of $n^2$ dephased matrices in canonical form with respect to permutation equivalence. A total order $<$ can be defined for the individual canonical forms, and the biggest (or smallest) of these can be taken as a canonical form of the complex Hadamard matrix. This procedure, which we will soon improve, is summarized in Algorithm~\ref{alg:canon1}. 

\begin{algorithm}
\caption{NaiveCanonicalForm$(H)$}
\label{alg:canon1}
\begin{algorithmic}
\State \textbf{external} CanPerm($H$)
\State $C\gets I$
\For{$a\in \{1,\ldots, n\}$}
\For{$b\in \{1,\ldots, n\}$}
\State $(\pi_{a,b},H_{a,b})\gets$ CanPerm($\mathscr{D}_{a,b}(H)$)
\If{$C=I$ or $C< H_{a,b}$ }
\State $(\sigma, C)\gets  (\pi_{a,b}\, \sigma_{a,b}, H_{a,b})$
\EndIf
\EndFor
\EndFor
\State \textbf{return} $(\sigma,C)$
\end{algorithmic}
\end{algorithm}

Computing canonical forms for the permutation graphs is the most computationally intensive step of Algorithm~\ref{alg:canon1}. We can avoid some of those computations by utilizing automorphisms of the complex Hadamard matrix. For $a,b\in \{1,\ldots,n\}$ and $\Gamma\subseteq S_n \times S_n$, we define an external function GetOrbit($a,b,\Gamma$) that returns the set $$\text{orb}_{\Gamma}(a,b)=\{(\pi_1(a),\pi_2(b))\;|\;\ (\pi_1,\pi_2)\in \langle \Gamma \rangle\},$$
that is, the \textit{orbit} of $(a,b)$ under $\langle \Gamma \rangle $. Before applying this function to improve Algorithm~\ref{alg:canon1}, we shall use it to compute the automorphism group of a complex Hadamard matrix.

We shall next develop an algorithm for finding generators of the group Aut($H$) for a given complex Hadamard matrix $H$. We use the notations $H_{a,b}:=\mathscr{D}_{a,b}(H)$ and $\pi_{a,b} \in S_n \times S_n$ for the pair of permutations such that $\pi_{a,b}\cdot H_{a,b}$ is in canonical form with respect to permutation equivalence.

The automorphisms of the permutation graph of $H_{1,1}$ give the stabilizer $$\text{stab}(1,1)=\{(\pi_1,\pi_2)\in \text{Aut}(H)\;|\;(\pi_1(1),\pi_2(1))=(1,1)\}$$ of the entry in row 1 and column 1, which is one subgroup in a stabilizer chain. The remaining automorphisms can be obtained by utilizing the following lemma.

\begin{lemma}
\label{new_lemma}
Let $H$ be a complex Hadamard matrix. Then
$
(a,b)\in \operatorname{orb}_{\operatorname{Aut}(H)}(1,1)
$
if and only if $H_{1,1}$ and $H_{a,b}$ are permutation equivalent.
\end{lemma}

\begin{proof}
Suppose that $(a,b)\in \operatorname{orb}_{\operatorname{Aut}(H)}(1,1)$. Then there exists an automorphism $\alpha=(\alpha_1,\alpha_2)$ such that $\alpha_1(1)=a$ and $\alpha_2(1)=b$. Set $\alpha'=\alpha^{-1}\sigma_{a,b}$. Then $\alpha'$ fixes $(1,1)$ and we can utilize equation~\eqref{commutes} to obtain
$$
H_{1,1}=\mathscr{D}(\alpha^{-1} \cdot H_{1,1})
=\mathscr{D}(\alpha'\,\sigma_{a,b}\cdot H_{1,1})
=\alpha'\cdot\mathscr{D}(\sigma_{a,b} \cdot H_{1,1})
=\alpha' \cdot H_{a,b}.
$$
For the converse, assume that $\pi_{1,1}\cdot H_{1,1}=\pi_{a,b}\cdot H_{a,b}$. Let
$$
\gamma_{a,b} :=\sigma_{a,b}\,\pi_{a,b}^{-1}\,\pi_{1,1}.
$$
Since
$
H_{a,b}=\pi_{a,b}^{-1}\,\pi_{1,1}\cdot H_{1,1}
$
and both $H_{1,1}$ and $H_{a,b}$ are dephased, the permutation
$\pi_{a,b}^{-1}\,\pi_{1,1}$ must fix $(1,1)$. Thus, $\gamma_{a,b}$ maps $(1,1)$ to $(a,b)$. Additionally, $\gamma_{a,b}$ is an automorphism, since
$$
\mathscr{D}(\gamma_{a,b}\cdot H)
=\mathscr{D}(\sigma_{a,b}\cdot (\pi_{a,b}^{-1} \, \pi_{1,1} \cdot H_{1,1}))
=\mathscr{D}(\sigma_{a,b} \cdot H_{a,b})
=\mathscr{D}(H).
$$
Therefore, $(a,b)\in \text{orb}_{\operatorname{Aut}(H)}(1,1)$.
\end{proof}

By Lemma~\ref{new_lemma}, we can now obtain the size of the orbit of the entry $(1,1)$ by computing the multiset $\{\pi_{a,b} \cdot H_{a,b}\}_{1 \leq a,b \leq n}$ and checking which of the matrices are identical to $\pi_{1,1}\cdot H_{1,1}$. For each such case, we can also compute the automorphism $\gamma_{a,b}$ appearing in the proof of Lemma~\ref{new_lemma}. Since the automorphisms $\gamma_{a,b}$ cover the entire orbit of the entry $(1,1)$, they form a \textit{left transversal} of $\operatorname{stab}(1,1)$. That is, each element of $\operatorname{Aut}(H)$ can be expressed as a composition of some $\gamma_{a,b}$ and some element of $\operatorname{stab}(1,1)$. The order of the automorphism group can be obtained with the orbit--stabilizer theorem, $$|\Aut(H)|=|\text{stab}(1,1)|\cdot|\mbox{orb}_{\text{Aut}(H)}(1,1)|.$$

We do not necessarily have to canonize all the matrices in $\{H_{a,b}\}_{1 \leq a,b \leq n}$. We can proceed iteratively over $a,b \in \{1,\ldots, n\}$ and for each $\gamma_{a,b}$ obtained, we can check if it can be used to enlarge the known orbit of (1,1). If we find a new element $(a',b')$ in the orbit of (1,1) it will save us from the canonization of the permutation graph of $H_{a',b'}$ as we know that it already has to be permutation equivalent to $H_{1,1}$. On the other hand, if $(a,b)$ does not belong to the orbit of $(1,1)$, then no element of orb$(a,b)$ can belong there either and we immediately disregard matrices $H_{a',b'}$, where $(a',b')\in\text{orb}(a,b)$.

Let $g_1,\ldots,g_l\in S_n\times S_n$ be the generators of stab$(1,1)$ obtained from the permutation graph of $\mathscr{D}(H)$. After iterating through all $a,b\in\{1,\ldots n\}$, we have obtained automorphisms $\gamma_1,\ldots,\gamma_d$. Now these elements generate the whole automorphism group,
$$\text{Aut}(H)=\langle g_1,\ldots,g_l,\gamma_1,\ldots,\gamma_{d} \rangle.$$ The whole procedure for determining the generators of Aut$(H)$ is outlined in Algorithm~\ref{aut-alg}.

\begin{algorithm}
\caption{AutomorphismGroup$(H)$}
\label{aut-alg}
\begin{algorithmic}
\State \textbf{external}  AutGen($G$), CanPerm($H$), GetOrbit($a,b,\Gamma$), PermGraph($H$)
\State $O\gets\{(1,1)\}$
\State $(\pi_{1,1},H_{1,1})\gets$ CanPerm($\mathscr{D}(H)$)
\State $\Gamma\gets$ AutGen(PermGraph$(\mathscr{D}(H)))$
\For{$a\in\{1,\ldots,n\}$}
\For{$b\in\{1,\ldots,n\}$}
\IfThen{$(a,b)\in O$}{\textbf{continue}}
\State $(\pi_{a,b},H_{a,b})\gets$ CanPerm($\mathscr{D}_{a,b}(H)$)
\If{$ H_{1,1} =  H_{a,b}$}
\State $\Gamma \gets\Gamma\cup\{\sigma_{a,b}\,\pi_{a,b}^{-1}\,\pi_{1,1}\}$
\State $O\gets O \cup$GetOrbit(1,1,$\Gamma$) 
\Else 
\State $O\gets O \cup$GetOrbit$(a,b,\Gamma)$
\EndIf
\EndFor
\EndFor
\State \textbf{return} $\Gamma$
\end{algorithmic}
\end{algorithm}

We can use the same idea of recording orbits when searching for the canonical form of $H$. In Algorithm~\ref{alg:canon-matrix}, we track the orbits already visited and accumulate automorphisms of $H$, which we then use to enlarge the orbits as the search proceeds.

\begin{algorithm}
\caption{CanonicalForm$(H)$}
\label{alg:canon-matrix}
\begin{algorithmic}
\State \textbf{external}  AutGen($G$), CanPerm($H$), GetOrbit($a,b,\Gamma$), PermGraph($H$)
\State $O\gets\{(1,1)\}$
\State $(\pi_{1,1},H_{1,1})\gets$ CanPerm($\mathscr{D}(H)$)
\State $(\sigma, C)\gets (\pi_{1,1}, H_{1,1})$
\State $\Gamma\gets$ AutGen(PermGraph$(\mathscr{D}(H)))$

\For{$a\in \{1,\ldots, n\}$}
\For{$b\in \{1,\ldots, n\}$}
\IfThen{$(a,b)\in O$}{\textbf{continue}}
\State $(\pi_{a,b},H_{a,b})\gets$ CanPerm($\mathscr{D}_{a,b}(H)$)
\If{$H_{1,1} = H_{a,b}$}
\State $\Gamma \gets\Gamma\cup\{\sigma_{a,b}\,\pi_{a,b}^{-1}\,\pi_{1,1}\}$
\EndIf
\If{$C<  H_{a,b}$}
\State $(\sigma, C)\gets (\pi_{a,b}\,\sigma_{a,b}, H_{a,b})$
\EndIf
\State $O\gets O\cup$GetOrbit$(a,b,\Gamma)$
\EndFor
\EndFor
\State \textbf{return} $(\sigma,C)$
\end{algorithmic}
\end{algorithm}

For completeness, we have also outlined Algorithm~\ref{alg:eq} for deciding if two complex Hadamard matrices $H_1$ and $H_2$ are equivalent. If $H_1\cong H_2$ the algorithm returns $\pi \in S_n \times S_n$ such that $\mathscr{D}(H_1)=\mathscr{D}(\pi\cdot H_2 )$.

\begin{algorithm}
\caption{AreEquivalent($H_1,H_2$)}
\label{alg:eq}
\begin{algorithmic}
\State \textbf{external}  CanonicalForm$(H)$
\State $(\sigma_1,C_1)\gets$CanonicalForm($H_1$)
\State $(\sigma_2,C_2)\gets$CanonicalForm($H_2$)
\If{$C_1=C_2$}
\State \textbf{return} $\sigma^{-1}_1\sigma_2$
\Else
\State \textbf{return} $\texttt{null}$
\EndIf
\end{algorithmic}
\end{algorithm}

\subsection{Invariants}
\label{inv}

An \emph{invariant} of a mathematical object is a property that
is the same for equivalent objects. Invariants can be utilized to
support algorithms for determining equivalence and automorphism
groups.

Many invariants have been proposed for complex Hadamard matrices.
For our purposes, we would like an invariant that is both 
sensitive and fast, which can be conflicting wishes. The
following invariant, introduced by Haagerup \cite{H97}, fulfills
our needs well. 

For two arbitrary rows and two arbitrary columns of a complex
Hadamard matrix of order $n$, $i,k$ and $j,l$, respectively, we define
\[
I(i,k,j,l) = h_{ij}\overline{h_{il}}\overline{h_{kj}}h_{kl}.
\]
Now the \emph{multiset}
\[
\bigcup_{\substack{1 \leq i,k \leq n, i\neq k \\ 1 \leq j,l \leq n, j \neq l}} I(i,k,j,l)
\]
is an invariant of the matrix. In the original paper \cite{H97},
the \emph{set} is considered, but a multiset is obviously more
sensitive. As transposition of $i$ and $k$ or $j$ and $l$
gives a conjugate of $I(i,k,j,l)$, 
we can speed up the calculation of this invariant
by a factor of 4 by considering the equivalent invariant obtained by requiring that $i < k$ and $j < l$, and 
taking the conjugate if the imaginary part is negative.

This matrix invariant can further be turned into an entry invariant.
In the case of entry invariants, different values mean that no automorphism maps one of the positions onto the other.
An entry invariant for row $a$, column $b$ is the multiset
\[
\bigcup_{\substack{1 \leq k \leq n, a\neq k \\ 1 \leq l \leq n, b \neq l}} I(a,k,b,l).
\]
The entry invariants can be obtained while computing
the matrix invariant, with some additional bookkeeping.

The Haagerup invariant is equivalent to the multiset of absolute values of $2\times 2$ minors as 
$$| h_{ij}h_{kl} - h_{il} h_{kj}|=|h_{ij}\overline{h_{il} h_{kj}}h_{kl}-1||h_{il} h_{kj}|=|I(i,k,j,l)-1|.$$
Including higher-order minors yields the more sensitive \emph{fingerprint} invariant \cite{S10}, but the computational cost is much higher.

Invariants can be used to speed up Algorithm~\ref{alg:eq} by deciding on the inequivalence of two complex Hadamard matrices immediately when the Haagerup invariants disagree. 

Entries with different values of an entry invariant cannot be in the same orbit. Hence, the Haagerup entry invariant can be used to reduce the number of permutation graphs for which canonical forms need to be found in Algorithm~\ref{aut-alg}  and~\ref{alg:canon-matrix}. Not only can we prune the search immediately when the entry invariants differ, but we can also restrict our consideration to the matrices with entry-invariant value that occurs least often (choosing the largest value in case of a tie).

\subsection{Symmetric and Hermitian matrices}
\label{SH}

Given a complex Hadamard matrix $H$, one can ask whether $H\cong \overline{H}$, $H\cong H^T$ or $H\cong H^\dagger$. These can all be directly decided with the algorithm developed.

One can also ask whether there exists a Hermitian or symmetric matrix equivalent to $H$. If $H'$ is Hermitian and $H \cong H'$, then $ H \cong H' = H'^\dagger \cong H^\dagger, $ so this is a stronger property than the previously considered $H \cong H^\dagger$. Similarly, if $H'$ is symmetric and $H \cong H'$, then $ H \cong H^T.$ 

We present two algorithms to test for these properties. Both rely on basic properties of symmetric and Hermitian matrices. A similar problem is considered in \cite{compWeight} in the context of weighing matrices on $q$th roots of unity.

\begin{lemma}
\label{lemma_1}
    Let $H$ be a complex Hadamard matrix. Then $H$ is equivalent to a Hermitian matrix if and only if there exists a matrix $M\in \mathcal{G}_n$ such that $MH$ is Hermitian.
\end{lemma}
\begin{proof}
Suppose that $XHY^\dagger$ is Hermitian. If we multiply on the left by $Y^\dagger$ and on the right by $Y$, we obtain $Y^\dagger XH$. Since this does not affect Hermiticity, setting $M=Y^\dagger X$ gives the desired result. The converse is immediate.
\end{proof}

Based on Lemma~\ref{lemma_1}, we can now process the matrix $M$ by a straightforward depth-first search described in Algorithm~\ref{alg1}.

\begin{algorithm}
\caption{IsEquivalentToHermitian(H)}
\label{alg1}
\begin{algorithmic}
\Procedure{DFS}{$i,H$}
  \IfThen{$i = n+1$}{\Return True}
  \For{$j \in \{i, \ldots, n\}$}
    \State $H'\gets ((i\;j),\text{id})\cdot H$.
    \State Multiply the $i$th row of $H'$ by $\overline{h_{1,i}h_{i,1}}$.
    \If{$h'_{i,k}h'_{k,i} =1$ for all $k\in \{1,\ldots,i\}$}
        \IfThen{DFS($i+1,H'$)}{\Return True}
    \EndIf
  \EndFor
  \State \Return False
\EndProcedure
\State $H \gets \mathscr{D}(H)$
\State \Return DFS(1,$H$)
\end{algorithmic}
\end{algorithm}

\begin{lemma}
\label{lemma_2}
     Let $H$ be a dephased complex Hadamard matrix. Suppose that for some pair of permutations $(\pi_1,\pi_2) \in S_n \times S_n$ we have $\mathscr{D}((\pi_1,\pi_2) \cdot H^\dagger) = H$. Then $H$ is equivalent to a Hermitian matrix if and only if there exists an automorphism $(\alpha_1,\alpha_2) \in \operatorname{Aut}(H)$ for which $\pi_1^{-1}\alpha_1=\alpha_2^{-1}\pi_2.$
\end{lemma}
\begin{proof}

    From Lemma~\ref{lemma_1}, it follows that $H$ is equivalent to a Hermitian matrix if and only if there exists $M\in \mathcal{G}_n$ such that $MHM=H^\dagger$. For necessity, assume that such an $M$ exists. Let $\sigma \in S_n$ be the permutation corresponding to the permutation part of $M$, so that $\mathscr{D}((\sigma, \sigma^{-1})\cdot H)=(M,M^\dagger)\cdot H$. Then setting $(\alpha_1,\alpha_2)=(\pi_1 \sigma, \pi_2 \sigma^{-1})$ gives the automorphism for which $\pi_1^{-1}\alpha_1=\alpha_2^{-1}\pi_2$. 
    
    Now, for the converse, suppose the existence of $(\alpha_1,\alpha_2) \in \operatorname{Aut}(H)$ satisfying the stated conditions. We have that
    $$\mathscr{D}((\pi_1^{-1}\alpha_1,\pi_2^{-1}\alpha_2) \cdot H)=\mathscr{D}((\pi_1^{-1},\pi_2^{-1})\cdot H)=H^\dagger.$$
    Hence, there exist matrices $D_1=\text{diag}(x_1,\ldots,x_n)$, $D_2=\text{diag}(y_1,\ldots,y_n)\in\mathcal{G}_n$ and a permutation matrix $P=P_{\pi_1^{-1}\alpha_1}=P_{\alpha_2^{-1}\pi_2}$ for which $D_1PHD_2=(PH)^\dagger$.  Setting $A=(a_{i,j})=PH$, we obtain the entrywise relation 
    $x_i a_{i,j}y_{j}=\overline{ a_{j,i}}$ implying that $x_iy_j =\overline{a_{i,j}a_{j,i}}$ for all $i,j\in \{1,\ldots,n\}$. Since the right-hand side is symmetric in $i$ and $j$, we have $x_i/y_i=x_j/y_j$ for all $i,j\in\{1,\ldots,n\}$. Therefore, these ratios are all equal to some constant $c\in S$. Thus, $D_1=cD_2$, and we obtain the desired $M$ by setting $M=\sqrt{c}D_2P$.
\end{proof}
 
Lemma~\ref{lemma_2} gives us an easy way to decide whether there exists an equivalent Hermitian matrix for a given complex Hadamard matrix by utilizing the methods developed. In Algorithm~\ref{alg2}, the procedure is outlined in terms of the algorithms from the previous section.

\begin{algorithm}[h]
\caption{IsEquivalentToHermitian2$(H')$}
\label{alg2}
\begin{algorithmic}
\State \textbf{external} AutomorphismGroup($H$), AreEquivalent($H_1,H_2$)
\State $H\gets \mathscr{D}(H')$
\State $(\pi_1,\pi_2)\gets$ AreEquivalent($H,H^\dagger$)
\IfThen{$(\pi_1,\pi_2)=\texttt{null}$}{\textbf{return} False}
\For{$(\alpha_1,\alpha_2)\in\langle$AutomorphismGroup($H$)$\rangle$}
\IfThen{$\pi_1^{-1}\alpha_1=\alpha_2^{-1}\pi_2$}{\textbf{return} True}
\EndFor
\State \textbf{return} False

\end{algorithmic}
\end{algorithm}

Algorithms~\ref{alg1} and~\ref{alg2} can, with minor adjustments, be adapted to consider symmetric matrices instead of Hermitian matrices.

\subsection{Implementation}
\label{implementation}

So far it has been assumed that the entries of complex Hadamard matrices are arbitrary 
unimodular complex numbers.
Unfortunately, the requirement of precise handling of arbitrary values enforces
excessively stringent restrictions
on implementations of the algorithms. Therefore, one may instead wish to implement them in
a numerical framework. Such an implementation in the C programming language
has been carried out by the authors and is available at \cite{codes}. Moreover, numerical 
methods are necessary whenever the entries of the complex Hadamard matrix are known 
only up to some precision, which is typically the case when matrices have been 
constructed with some iterative approach. We shall now have a brief look at some of 
the central issues in a numerical implementation.


\paragraph{Representation of numbers} A unimodular complex number 
$s = \exp(\theta i)$ is represented through its argument $\theta$, 
where $-\pi < \theta \leq \pi$. 

\paragraph{Comparison of numbers} A core subtask of the numerical computations
is that of determining whether two unimodular complex numbers are identical or
different. For such a comparison, we define the distance between 
$s_1=\exp(\theta_1 i)$ and $s_2=\exp(\theta_2i)$ as the shortest arc length 
on the unit circle between them, formally
 $$
 d(s_1,s_2)=d'(\theta_1,\theta_2)=\min(|\theta_2-\theta_1|,\;2\pi-|\theta_2-\theta_1|)
 .$$
With respect to a general parameter $\Delta > 0$, values are now
considered identical whenever $d(s_1,s_2) < \Delta$. The choice of a proper
value of $\Delta$ is easy if we know in advance a lower bound on 
$d(s_1,s_2)$, $s_1 \neq s_2$, which is the case, for example, for Butson-type
Hadamard matrices. Generally, a useful guideline is to choose $\Delta$ slightly larger than the precision of the entries in the input matrix.

\paragraph{Numerical accuracy} With unimodular complex numbers represented
by their arguments, multiplication is done by adding an argument and
multiplication by a complex conjugate is done by subtracting an argument.
Each entry is subject to only a few additions and subtractions, so the
impact of the computations on the accuracy is negligible. 

\paragraph{Canonical form} There are many technical issues in defining and
implementing a numerical canonical form; let us just highlight one major
point. Whereas the distance function correctly sees that $\pi - \epsilon$ is
close to $-\pi+\epsilon$, care must be taken when sorting is carried out
based on values of elements. This task gets easier if values of $\theta$ in
the interval $-\pi < \theta < -\pi + \Delta$ are replaced by $\pi$.
Specifically, coloring of entry vertices in the permutation graph can then be
done setwise with colors $1,2,\ldots$ in the order from the smallest to 
the largest values.

\vspace{0.5cm}
A numerical approach can still be used as a tool for obtaining analytical
results. Namely, one can analytically (try to) prove existence of mappings 
obtained in computations of automorphism groups and equivalence. This
approach was used extensively and successfully in this work, as we shall
see in Section~\ref{S3}.

\section{Applications}
\label{S3}

\subsection{Specific complex Hadamard matrices}

For the Fourier matrix $F_n=[\omega^{(j-1)(k-1)}]_{j,k}$, where $\omega=\exp(2\pi i/n)$, the automorphism group can be deduced analytically.

\begin{theorem}
    The automorphism group of a Fourier matrix of order $n$ is isomorphic to
    $$(C_n\times C_n)  \rtimes (\mathbb{Z}/n\mathbb{Z})^\times.$$
\end{theorem}
\begin{proof}
    In \cite{H19}, it has been shown that $\text{stab}(1,1)\leq \Aut(F_n)$ equals $$\{(\pi_j,\pi_j^{-1}) \;|\; \gcd(j,n) = 1\} \cong (\mathbb{Z}/n\mathbb{Z})^\times,$$ where $\pi_j(k)=(j(k-1)\mod n) + 1.$ Let $\sigma_k = (1 \;2\; \cdots\; n)^k.$ We have that 
    $\mathscr{D}((\sigma_a,\sigma_b)\cdot F_n)=F_n$ for all $a,b\in \{0,\ldots,n-1\}$.
Thus, $(\sigma_a,\sigma_b)$ is always an automorphism of $F_n$. Moreover, the group $$N\coloneqq\{(\sigma_a,\sigma_b) \;|\; a,b\in \{0,\ldots,n-1\} \} \cong C_n\times C_n$$ acts regularly on the entries of $F_n$ so it forms a left transversal of stab$(1,1)$. Hence, every element in Aut$(F_n)$ can be expressed as $(\sigma_a, \sigma_b)(\pi_j,\pi^{-1}_j)$ for some $a,b\in\{0,\ldots ,n-1\}$ and $j\in (\mathbb{Z}/n\mathbb{Z})^\times$. In addition, $$(\pi_j ,\pi_j^{-1})( \sigma_a, \sigma_b)(\pi_j^{-1}, \pi_j) = (\sigma_{aj},\sigma_{bj^{-1}})\in N,$$ so $N$ is a normal subgroup of Aut$(F_n)$. We have Aut$(F_n)=N\hspace{1pt}\text{stab}(1,1)$, $N\cap \text{stab}(1,1)=\{\operatorname{id}\}$ and $N\trianglelefteq \operatorname{Aut}(F_n)$, so we can write Aut$(F_n)$ as a semidirect product of $N$ and stab$(1,1)$.
\end{proof}

Automorphisms of real Hadamard matrices have been studied extensively. For small orders, automorphism groups of real Hadamard matrices can be found in \cite{P08}. There are also results for Butson-type matrices that provide explicit groups \cite{GPP18, M01} and classification related work \cite{EFO15, LOS20, LSO13} that give more general results. To further validate correctness of the implementation of our algorithm, we compared its output with the available data. In the cases we checked, the results coincide.

In Table \ref{T1}, automorphism groups for various complex Hadamard matrices are presented. The necessary accuracy considerations have been made and the results of the computation can be taken as exact. To the best of the authors' knowledge, only the groups corresponding to real Hadamard matrices and some small-order Butson-type Hadamard matrices have appeared in the literature. The automorphism group of the complex Hadamard matrix $S_6$ is obtained in \cite{GPP18} as a byproduct of computing the outer automorphism of the symmetric group on six
points. 

Most of the matrices are gathered from \cite{B06}, and the same naming conventions are used. In particular, Walsh matrices are denoted by $H_n$ and Butson-type matrices from \cite{H28ref,LSO17} are denoted by BH$(n,q)_s$, where $s$ is their index in the corresponding database. 

 The column Structure contains a structural description of the computed automorphism group given by GAP \cite{GAP4} in short format (cyclic groups are denoted with their orders). Times presented are the times taken for determining the automorphism group. The algorithms were run on an HP EliteBook 840 G11 laptop with an Intel Core Ultra 5 125U CPU. The last column answers the five questions whether $H\cong \overline{H}$, $H\cong H^T$, $H\cong H^\dagger$, whether $H$ is equivalent to a symmetric matrix, and whether $H$ is equivalent to a Hermitian matrix in this order; Y for yes and N for no.

{\setlength{\tabcolsep}{4pt} 
\begin{longtable}{|llrlrc|}

\caption{Automorphism groups} \label{T1} \\
\toprule
Name & Ref.&  $|\text{Aut}|$ & Structure  & Time & $\cong$ \\
\midrule
\endfirsthead

\multicolumn{6}{c}%
{{\tablename\ \thetable{} (cont.)}} \\
\toprule
Name & Ref.&  $|\text{Aut}|$ & Structure  & Time & $\cong$ \\ 
\midrule
\endhead
\hline
\endfoot

\hline \hline
\endlastfoot

$F_2$ & - & 4 & \small $2^2$  & 0.03 ms & \small YYYYY \\ 
$F_3$ & - & 18 & \small $(3^2):2$  & 0.03 ms & \small YYYYN \\ 
$F_4$ & - & 32 & \small $(4^2):2$  & 0.02 ms & \small YYYYN \\ 
$H_4$ & - & 96 & \small $((2^4):3):2$  & 0.02 ms & \small YYYYY \\ 
$F_5$ & - & 100 & \small $(5^2):4$  & 0.04 ms & \small YYYYN \\ 
$F_6$ & - & 72 & \small $2^2 \times ((3^2):2)$  & 0.05 ms & \small YYYYN \\ 
$F_2\otimes F_3$ & - & 72 & \small $2^2 \times ((3^2):2)$  & 0.05 ms & \small YYYYN \\ 
$S_6$ & \cite{S6} & 360 & \small $A6$  & 0.04 ms & \small YYYYN \\ 
$D_6$ & \cite{H97} & 60 & \small $A5$  & 0.04 ms & \small YYYYY \\ 
$C_6$ & \cite{H97} & 12 & \small $D12$  & 0.07 ms & \small YYYYY \\ 
$P_7$ & \cite{P7} & 12 & \small $D12$  & 0.08 ms & \small NYNYN \\ 
$F_7$ & - & 294 & \small $7^2:6$  & 0.06 ms & \small YYYYN \\ 
$C_{7A}$ & \cite{H97} & 168 & \small $PSL(3,2)$  & 0.06 ms & \small NYNYN \\ 
$C_{7B}$ & \cite{H97} & 168 & \small $PSL(3,2)$  & 0.06 ms & \small NYNYN \\ 
$C_{7C}$ & \cite{H97} & 14 & \small $D14$  & 0.09 ms & \small NYNYN \\ 
$C_{7D}$ & \cite{H97} & 14 & \small $D14$  & 0.08 ms & \small NYNYN \\ 
$Q_7$ & \cite{S11} & 6 & \small $6$  & 0.06 ms & \small NYNYN \\ 
$F_8$ & - & 256 & \small $8^2:2^2$  & 0.09 ms & \small YYYYN \\ 
$S_8$ & \cite{S8} & 48 & \small $2 \times S4$  & 0.08 ms & \small YNNNN \\ 
$D_{8A}$ & \cite{D8} & 32 & \small $(2^3):(2^2)$  & 0.10 ms & \small YNNNN \\ 
$V_8$ & \cite{B06} & 16 & \small $D16$  & 0.12 ms & \small NYNYN \\ 
$A_8$ & \cite{A8} & 48 & \small $GL(2,3)$  & 0.11 ms & \small NYNYN \\ 
$H_8$ & - & 10752 & \small $2^6:PSL(3,2)$  & 0.05 ms & \small YYYYY \\ 
$F_2 \otimes F_4$ & - & 512 & \small $(2^2 \times ((2^3):(2^2))):2^2$  & 0.08 ms & \small YYYYN \\ 
$F_9$ & - & 486 & \small $9^2:6$  & 0.15 ms & \small YYYYN \\ 
$S_9$ & \cite{S9} & 36 & \small $2 \times ((3^2):2)$  & 0.16 ms & \small NNYNN \\ 
$B_9$ & \cite{B6} & 18 & \small $(3^2):2$  & 0.25 ms & \small YYYYN \\ 
$N_9$ & \cite{B6} & 12 & \small $D12$  & 0.26 ms & \small NYNYN \\ 
$Y_9$ & \cite{TB} & 8 & \small $8$  & 0.12 ms & \small YYYYN \\ 
$F_3 \otimes F_3$ & - & 3888 & \small $((3^4:Q8):3):2$  & 0.11 ms & \small YYYYY \\ 
$D_{10}$ & \cite{D10} & 720 & \small $S6$  & 0.13 ms & \small YYYYY \\ 
$F_{10}$ & - & 400 & \small $2^2 \times ((5^2):4)$  & 0.16 ms & \small YYYYN \\ 
$N_{10B}$ & \cite{B06} & 16 & \small $(4 \times 2):2$  & 0.40 ms & \small NYNYN \\ 
$G_{10}$ & \cite{LSO13} & 5 & \small $5$  & 0.44 ms & \small YYYYY \\ 
$N_{10A}$ & \cite{B6} & 72 & \small $(S3 \times S3):2$  & 0.23 ms & \small YYYYN \\ 
$S_{10}$ & \cite{S10} & 50 & \small $(5^2):2$  & 0.45 ms & \small YYYYN \\ 
$F_2\otimes F_5$ & - & 400 & \small $2^2 \times ((5^2):4)$  & 0.24 ms & \small YYYYN \\ 
$F_{11}$ & - & 1210 & \small $11^2:10$  & 0.18 ms & \small YYYYN \\ 
$C_{11A}$ & \cite{H97} & 660 & \small $PSL(2,11)$  & 0.17 ms & \small NYNYN \\ 
$N_{11A}$ & \cite{N11A} & 10 & \small $D10$  & 0.65 ms & \small NYNYN \\ 
$N_{11B}$ & \cite{B6} & 10 & \small $D10$  & 0.26 ms & \small NYNYN \\ 
$Q_{11}$ & \cite{S11} & 20 & \small $D20$  & 0.19 ms & \small YYYYN \\ 
$\small \text{BH}(12,2)$ & \cite{H12} & 95040 & \small $M12$  & 0.13 ms & \small YYYYY \\ 
$C_{13A}$ & \cite{H97} & 78 & \small $13:6$  & 0.22 ms & \small YYYYN \\ 
$C_{13B}$ & \cite{H97} & 78 & \small $13:6$  & 0.21 ms & \small YYYYN \\ 
$S_{14}$ & \cite{S9} & 196 & \small $7^2:4$  & 1.14 ms & \small YYYYN \\ 
$L_{14A}$ & \cite{LSO13} & 1 & \small $1$  & 3.99 ms & \small NNNNN \\ 
$H_{16}$ & - & 5160960 & \small $2^8:A8$  & 0.28 ms & \small YYYYY \\ 
$\small\text{BH}(21,3)_3$ & \cite{LSO17} & 240 & \small $2 \times S5$  & 6.55 ms & \small NYNYN \\ 
$\small\text{BH}(21,3)_7$ & \cite{LSO17} & 8 & \small $2^3$  & 27.86 ms & \small NNNNN \\ 
$\small\text{BH}(21,3)_{19}$ & \cite{LSO17} & 168 & \small $PSL(3,2)$  & 4.90 ms & \small NNNNN \\ 
$\small\text{BH}(28,2)_{109}$ & \cite{H28ref} & 168 & \small $(14 \times 2):6$  & 4.11 ms & \small YNNNN \\ 
$\small \text{BH}(28,2)_{278}$ & \cite{H28ref} & 28 & \small $D28$  & 43.30 ms & \small YNNNN \\ 
$H_{32}$ & - & $2^{20}\cdot 9765$  & \small $2^{10}:PSL(5,2)$ & 0.93 ms & \small YYYYY\\
$B_{36}$ & \cite{B36} & 144 & \small $S3 \times S4$  & 6.54 ms & \small NNNNN \\ 
$P_7\otimes S_6$ & - & 4320 & \small $2 \times S3 \times A6$  & 5.11 ms & \small NYNYN \\ 
$A_8\otimes V_8$ & - & 768 & \small $GL(2,3) \times D16$  & 35.26 ms & \small NYNYN \\ 
$F_{71}$ & - & 352870 & \small $71^2:70$  & 9.70 ms & \small YYYYN \\

\bottomrule
\end{longtable}}

\subsection{Families of complex Hadamard matrices}

In the literature, complex Hadamard matrices often appear as continuous families.  For example,
    $$F_4(a)=\left[\begin{array}{rrrr}
1&  1        &  1&          1\\
1&  ia& -1& -ia\\
1& -1        &  1&         -1\\
1& -ia& -1&  ia
\end{array}\right]$$
gives a complex Hadamard matrix for all $a\in S$. 

Instead of computing properties of a specific family member it can be more interesting to find properties that hold for all matrices in the family. For a family of complex Hadamard matrices $\mathcal{H}$, we define the group $\bigcap_{H\in\mathcal{H}}\text{Aut}(H)$ as the \emph{generic automorphism group} of that family.

We can obtain a generic automorphism group by computing the automorphism group for a randomly chosen family member, and then showing analytically that the permutations in the group are automorphisms for every matrix in the family.

In Table \ref{T2}, generic automorphism groups for various families of complex Hadamard matrices are presented. The groups have been analytically checked to stabilize all the matrices in that family. The naming of the columns follows that of Table~\ref{T1}.

{\setlength{\tabcolsep}{4pt} 
\begin{longtable}{|llrlr|}
\caption{Generic automorphism groups of families} \label{T2} \\
\toprule
Name & Ref.&  $|\text{Aut}|$ & Structure  & Time  \\
\midrule
\endfirsthead

\multicolumn{5}{c}%
{ \tablename\ \thetable{} (cont.)} \\
\toprule
Name & Ref.&  $|\text{Aut}|$ & Structure  & Time  \\ 
\midrule
\endhead
 \hline
\endfoot

\hline \hline
\endlastfoot

$F_4^{(1)}$ & \cite{H12} & 16 & \small $2 \times D8$  & 0.04 ms  \\ 
$F^{(2)}_6$ & \cite{H97} & 6 & \small $6$  & 0.06 ms  \\ 
$D^{(1)}_6$ & \cite{D6} & 12 & \small $A4$  & 0.06 ms  \\ 
$B^{(1)}_6$ & \cite{B6} & 3 & \small $3$  & 0.07 ms  \\ 
$M_6^{(1)}$ & \cite{M6} & 2 & \small $2$  & 0.05 ms  \\ 
$X_6^{(2)}$ & \cite{X6} & 3 & \small $3$  & 0.06 ms  \\ 
$K_6^{(2)}$ & \cite{K62} & 2 & \small $2$  & 0.05 ms  \\ 
$K_6^{(3)}$ & \cite{K63} & 1 & \small $1$  & 0.06 ms  \\ 
$G_6^{(4)}$ & \cite{G6} & 1 & \small $1$  & 0.06 ms  \\ 
$P_7^{(1)}$ & \cite{P7} & 2 & \small $2$  & 0.08 ms  \\ 
$F^{(5)}_{8}$ & \cite{TZ06} & 8 & \small $2^3$  & 0.13 ms  \\ 
$S^{(4)}_{8}$ & \cite{S8} & 1 & \small $1$  & 0.20 ms  \\ 
$D^{(5)}_{8A}$ & \cite{D8} & 8 & \small $2^3$  & 0.13 ms  \\ 
$T^{(3)}_{8B}$ & \cite{TB} & 8 & \small $2^3$  & 0.17 ms  \\ 
$T^{(3)}_{8C}$ & \cite{V25} & 4 & \small $4$  & 0.14 ms  \\ 
$T^{(3)}_{8D}$ & \cite{V25} & 4 & \small $2^2$  & 0.15 ms  \\ 
$T^{(3)}_{8E}$ & \cite{V25} & 4 & \small $2^2$  & 0.24 ms  \\ 
$T^{(3)}_{8F}$ & \cite{V25} & 4 & \small $2^2$  & 0.16 ms  \\ 
$F^{(3)}_{9}$ & \cite{TZ06} & 27 & \small $3^3$  & 0.12 ms  \\ 
$K^{(2)}_9$ & \cite{K9} & 18 & \small $(3^2):2$  & 0.18 ms  \\ 
$D^{(3)}_{10}$ & \cite{D10} & 4 & \small $2^2$  & 0.34 ms  \\ 
$F^{(4)}_{10}$ & \cite{TZ06} & 10 & \small $10$  & 0.20 ms  \\ 
$G^{(1)}_{10}$ & \cite{LSO13} & 5 & \small $5$  & 0.16 ms  \\ 
$N^{(3)}_{10B}$ & \cite{LSO13} & 4 & \small $2^2$  & 0.58 ms \\ 
$H^{(7)}_{12}$ & \cite{D10} & 1 & \small $1$  & 2.09 ms  \\ 
\bottomrule
\end{longtable}}

\section*{Acknowledgements}
The second author acknowledges financial support of the Finnish Ministry of Education and Culture through the Quantum Doctoral Education Pilot Program (QDOC VN/3137/2024-OKM-4) and the Research Council of Finland through the Finnish Quantum Flagship project (Aalto 358877).

\section*{Data availability}
The data supporting the findings of this study are publicly available at \cite{codes}.

\section*{Declarations}
\paragraph{Conflict of interest} The authors have no relevant financial or non-financial interests to disclose.

\end{document}